\documentclass[11pt]{article}
\usepackage{cite}
\usepackage{fullpage}
\usepackage{amssymb}
\usepackage{amsthm}
\usepackage{amsmath}
\usepackage{graphicx}
\usepackage[vcentermath]{youngtab}
\usepackage[all,2cell]{xy}
\usepackage{listings}
\usepackage{young}
\usepackage{enumitem}
\usepackage{youngtab}
\usepackage{ytableau}
\usepackage{tikz}
\usepackage[colorlinks=true, linkcolor=blue, citecolor=blue, urlcolor=blue]{hyperref}
\usepackage{varwidth}

\newtheorem{theorem}{Theorem}[section]

\newtheorem{example}[theorem]{Example}
\newtheorem{lemma}[theorem]{Lemma}

\newtheorem{remark}[theorem]{Remark}

\newcommand{\eat}[1]{}

\newcommand{\taumn}[2]{w_{#1,#2}}
\newcommand{\gmodp}{G/P}

\newcommand{\grass}[2]{\mathrm{Gr}_{{#1},{#2}}}

\newcommand{\xtaumodt}[2]{T\backslash\mkern-6mu\backslash X(\taumn{#1}{#2})}
\newcommand{\schbmodt}[1]{T\backslash\mkern-6mu\backslash X(#1)}
\begin{document}
\title{Smooth torus quotients of Richardson varieties in the Grassmannian}
\author{
Sarjick Bakshi
\thanks{Indian Institute of Technology, Bombay {\tt sarjick91@gmail.com}}
}

\date{}
\maketitle

\begin{abstract}

Let $k$ and $n$ be positive coprime integers with $k<n$. Let $T$ denote the subgroup of diagonal matrices in $SL(n,\mathbb{C})$. We study the GIT quotient of Richardson varieties $X^v_w$ in the Grassmannian $\grass{k}{n}$ by $T$ with respect to a $T$-linearised line bundle $\cal{L}$ corresponding to the Pl\"{u}cker embedding. We give necessary and sufficient combinatorial conditions for the quotient variety $T \backslash\mkern-6mu\backslash (X_w^v)^{ss}_T({\cal L})$ to be smooth.

\end{abstract}

\section{Introduction}\label{r.introduction}

A {\em Richardson variety} is the intersection of a {\em Schubert variety} and an {\em opposite Schubert variety} in a {\em generalised flag variety}. They first appeared in the work of Richardson \cite{richardson1992intersections}. Since its introduction there have been many applications of Richardson varieties. For instance, they were considered by Lakshmibai--Brion \cite{brion2003geometric} to obtain a geometric construction of a standard monomial theoretic basis of the global sections of line bundle on flag variety. They also appeared in the studies of $K$-theory for the flag varieties (\cite{brion2002positivity, lakshmibai2003richardson}). In this paper, we will study {\em geometric invariant theoretic (GIT)} quotients of Richardson varieties in the Grassmannian by a maximal torus for a polarised line bundle and obtain a combinatorial description of the smooth quotients.

Let $\grass{k}{n}$ denote the {\em Grassmannian variety} of $k$-subspaces in complex $n$-space. Let $G=SL(n,\mathbb{C})$. Let $T$ be the subgroup of all diagonal matrices in $G$, and $B$ the subgroup of all upper triangular matrices in $G$ and $B^{-}$ the subgroup of all lower triangular matrices in $G$. We know that $G$ acts transitively on $\grass{k}{n}$ by left multiplication. Let $e_1,\ldots,e_n$ be the standard basis of $\mathbb{C}^n$. Let $P$ be the stabiliser $ \langle e_1,e_2,\ldots,e_r \rangle$ in $G$. So the Grassmannian variety can also be realised as the homogeneous space $G/P$. We note that $P$ is a parabolic subgroup containing $B$. Let $W_P$ denote the {\em Weyl group} of $P$. Let $W^P = W/W_P$ denote the set of {\em minimal length coset representative}. The $T$-fixed points in $G/P$ are $e_w=wP/P$ with $w \in W^{P}$. The $B$-orbit $C_w$ of $e_w$ is called a {\em Schubert cell} and it is an affine space of dimension $l(w)$. The closure of $C_w$ in $\gmodp$ is the {\em Schubert variety $X(w)$}. Dually, one defines the {\em opposite Schubert cell $C^v$} to be the $B^{-}$-orbit through $e_v$ and the closure of the opposite Schubert cell is called the {\em opposite Schubert variety}. The {\em Richardson variety $X^v_w$} is the intersection of the Schubert variety $X(w)$ with the opposite Schubert variety $X^v$. It is shown to be non empty if and only if $v \leq w$ in the Bruhat order. Let $S_d$ denote the symmetric group in $d$ letters. We observe that $W_{P}= S_k \times S_{n-k}$,  so the minimal length coset representatives of $W^P$ 
can be identified with \[\{w \in W | w(1) < w(2) < \ldots < w(k), w(k+1) < w(k+2) < \ldots < w(n)\}.\] Let \[I(k,n) = \{(i_1,i_2,..i_k) | 1 \leq i_1 < i_2  \cdots <i_k \leq n \}.\]Then there is a natural identification of  $W^P$ with $I(k,n)$ sending $w$ to $(w(1),w(2),\ldots,w(k))$.

In \cite{kreiman2002richardson}, Lakshmibai--Kreiman gave a self-contained presentation of standard monomial theory for unions of Richardson varieties in the Grassmannian. In the same paper, they determine a basis for the tangent space and gave a criteria for smoothness for $X^v_w$ at any $T$-fixed point $e_{\tau}$. Billey--Coskun \cite{billey2012singularities} introduces a generalisation of Richardson varieties called the {\em intersection varieties}. They 
characterize the smooth Richardson varieties in terms of vanishing conditions on certain products of cohomology classes for Schubert varieties. Richardson varieties in the Grassmannian are also studied in \cite{stanley1977some}, where these varieties are called {\em skew Schubert varieties}.  

Let $p: \grass{k}{n} \hookrightarrow \mathbb{P}(\bigwedge^k\mathbb{C}^n)$ denote the {\em Pl\"{u}cker embedding} where a $k$-dimensional vector space is sent to its $k$ wedges. Let $\cal{M}$ denote the pullback $p^*{\cal{O}}(1)$. We note that $\cal{M}$ is a $T$-linearised line bundle on $\grass{k}{n}$. In \cite{kumar2008descent}, Kumar showed that ${\cal{L}}:={\cal{M}}^n$ descends to the GIT quotient of the $\grass{k}{n}$ by the maximal torus $T$. Let $({\grass{k}{n}})^{ss}_T({\cal L})$(respectively, $({\grass{k}{n}})^{s}_T({\cal L})$) denote the set of all semistable (respectively, stable) points with respect to the $T$-linearized line bundle ${\cal L}$. When $k$ and $n$ are coprime, Skorobogatov \cite{skorobogatov1993swinnerton} and independently Kannan\cite{kannan1998torus} have shown that all semistable points are stable. They further prove that the GIT quotient $T\backslash \backslash (\grass{k}{n}^{ss})_T({\cal{L}})$ is smooth.

In \cite{kannan2009torusA}, Kannan--Sardar showed that there is a unique minimal Schubert Variety $X(w_{k,n})$ in $\grass{k}{n}$ admitting semistable points with respect to the line bundle $\mathcal{L}$ whenever $k$ and $n$ are coprime. However, there was a computational error in their final result which has been corrected in Bakshi--Kannan--Subrahmanyam \cite{bakshi2019torus} while obtaining a simpler proof. Let $w_{k,n} = (a_1,a_2,\ldots,a_k)$. In \cite{kannan2018torus}, Kannan--Paramasamy--Pattanayak--Upadhyay gave a condition on $v$ and $w$ for which the semistable locus in the Richardson variety $X^v_w$ in $\grass{k}{n}$ is non empty. The GIT quotient $\xtaumodt{k}{n}^{ss}_{T}({\cal L})$ is shown to be smooth in \cite{bakshi2019torus} by Bakshi--Kannan--Subrahmanyam. In the same paper, they showed that there exist a unique minimal Schubert variety $X^{v_{k,n}}_{w_{k,n}}$ in $\grass{k}{n}$ admitting semistable points with respect to the line bundle $\mathcal{L}$. They show $v_{k,n} = (1,a_1,\ldots, a_{k-1})$. In  \cite{bakshiinprep}, Bakshi--Kannan--Subrahmanyam gave a combinatorial criteria on $w$ for which the GIT quotient $T \backslash\mkern-6mu\backslash X(w)^{ss}_T({\cal L})$ is smooth. More recently, a similar criteria is also obtained by Chary--Pattanayak \cite{bonala2021torus}. 

In this paper, we study the GIT quotients of Richardson varieties $X_w^v$ in the Grassmannian generalising the result of Bakshi--Kannan--Subrahmanyam in \cite{bakshiinprep}. We obtain a condition on $v$ and $w$ for which $T \backslash\mkern-6mu\backslash (X_w^v)^{ss}_T({\cal L})$ is smooth. Our main result is the following:

\begin{theorem} \label{maintheorem}
Let $k$ and $n$ be coprime. Let $v = (c_1, c_2,\ldots, c_k)$ and $w = (b_1,b_2,\ldots, b_k)$ be such that $v \leq v_{k,n}$ and $w \geq w_{k,n}$. \eat{ Let $c_1 = 1$. Let $b_i \geq a_i$ for all $1 \leq i \leq n$ and $c_i \leq a_{i-1}$ for all $2 \leq i \leq n$.} Let $X_{w_1}^{v_1},\ldots, X_{w_r}^{v_r}$ be $r$ components in the singular locus of $X(w)$. Then the following are equivalent
\begin{itemize}
 \item[$(1)$] $T \backslash\mkern-6mu\backslash (X_w^v)^{ss}_T({\cal L})$ is smooth.
 \item[$(2)$] We have $ w_i \ngeqslant w_{k,n}$ and $ v_i \nleqslant v_{k,n} $ for all $i$. 
 \item[$(3)$] Whenever $b_j \geq b_{j-1}+2$, we have $a_j \geq b_{j-1}+1$, and whenever $c_j \geq c_{j-1}+2$, we have $ a_{j-1} \leq c_j+1$.
 \end{itemize} 
\end{theorem}

\section{Acknowledgment}

The author is grateful to Senthamarai Kannan and K V Subrahmanyam for carefully reading the paper and providing valuable comments and suggestions. The author would also like to thank Dipendra Prasad for constant encouragement. The author is supported by a postdoctoral fellowship at Indian Institute of Technology, Bombay. 
 
\section{Singular locus of Richardson varieties}

Let $1 \leq k <  n$. Let $\grass{k}{n}$ be the {\em Grassmannian variety} of $k$ dimensional subspace in $n$ dimensional complex vector space. We recall that the {\em Pl\"{u}cker embedding}
\[ p: \grass{k}{n} \longrightarrow \mathbb{P}(\bigwedge^k\mathbb{C}^n) \]
takes a $k$-dimensional vector space $V$ and maps it to $[\wedge^k(V)]$. 

The map $p$ gives a closed embedding of $\grass{k}{n}$ inside $\mathbb{P}(\bigwedge^k\mathbb{C}^n)$, hence giving a projective variety structure to the Grassmannian. Let 
\[ I(k,n) = \{(i_1,i_2,..i_k) | 1 \leq i_1 < i_2  \cdots <i_k \leq n \}.  \]

We recall that $G=SL(n,\mathbb{C})$ acts transitively on $\grass{k}{n}$. Denoting $e_1,\ldots,e_n$ as the standard basis of $\mathbb{C}^n$ we observe that the stabiliser of $\langle e_1,e_2,\ldots,e_k \rangle$ is given by the subgroup $P = \begin{bmatrix}
          * & * \\
	    0_{n-k,k}  &  * 
          \end{bmatrix}$. Thus $\grass{k}{n}$ gets identified with $G/P$. Let $T$ be a {\em maximal torus} consisting of the diagonal matrices in $G$. Let $B$ denote the {\em Borel subgroup} consisting of all upper triangular matrices in $G$ and $B^{-}$ the {\em opposite Borel subgroup} consisting of all lower triangular matrices in $G$. Let $S_d$ denote the symmetric group in $d$ letters. We know that the {\em Weyl group} of $G$ with respect to $T$ is $S_n$. The Weyl group $W_P$ of the parabolic subgroup $P$ is given by $S_k \times S_{n-k}$. The set of {\em minimal length coset representatives} $W^P = W/W_P$ gets identified with the set $I(k,n)$: $w \mapsto (w(1),w(2),\ldots,w(r))$. For $w \in I(k,n)$, let $C_w := BwP/P$(respectively, $C^w := B^{-}wP/P$) denote the {\em Schubert cell} (respectively, {\em opposite Schubert cell}). The {\em Schubert variety $X(w)$}(respectively, {\em opposite Schubert variety $X^w$}) is $\overline{C_w}$(respectively, $\overline{C^w}$) inside $\grass{k}{n}$. For $v,w \in I(k,n)$, the {\em Richardson variety $X_w^v$} is defined as $X(w) \cap X^v$. Note that this intersection is a scheme-theoretic intersection. 
          
Let $v,w \in I(k,n)$. Let $w = (w_1, w_2, \ldots w_k)$ and $v = (v_1,v_2,\ldots v_k)$. We define a partial order $\leq$ on $I(k,n)$ as $v \leq w$ if and only if $v_t \leq w_t$ for all $1 \leq t \leq k$. This order is called the {\em Bruhat order} on $I(r,n)$. We recall from \cite{kreiman2002richardson} that $X_w^v$ is non empty if and only $v \leq w$. One also notes that $X^{v_1}_{w_1} \subset X^{v_2}_{w_2} $ if and only if $w_1 \leq w_2$ and $v_1 \geq v_2$ in the Bruhat order.

The singular loci of Schubert varieties in $\grass{k}{n}$ were determined by Lakshmibai and Weyman~\cite{lakshmibai1990multiplicities}. We briefly recall their description of the singular locus. Let $w = (w_1,w_2,\ldots w_k) \in I(k,n)$.  We associate to each $w$ a weakly increasing sequence $\bf{w} = ({ \bf w_1},{\bf w_2},\ldots,{\bf w_k})$ where ${\bf w_i} = w_{i}-i$, so $0 \leq {\bf w_1} \leq {\bf w_2} \leq \ldots \leq {\bf w_k} \leq n-k$. Each such  $\bf{w}$ can be represented by Young diagram ${\cal{Y}}(w)$, in a $k \times (n-k)$ rectangle by putting ${\bf w_i}$ boxes in the $i$-th row from bottom. Note the filling of the boxes are from left to right.

Recall the following theorem from \cite{lakshmibai1990multiplicities}\footnote{Note that the notation we use is different from that used in \cite{lakshmibai1990multiplicities}, they work with non-increasing sequences.}.
\begin{theorem}[Theorem 5.3 \cite{lakshmibai1990multiplicities}] 
\label{LW90}
Let $X(w)$ be a Schubert variety in the Grassmannian. 
Let $\bf{w} = (p_1^{q_1},\ldots, p_r^{q_r}) = (\underbrace{p_1,\ldots p_1}_\text{$q_1$ times}\ldots,\underbrace{p_r,\ldots p_r}_\text{$q_r$ times})$  be the non-zero parts of the increasing sequence ${\bf w}$ with 
$1 \leq  {\bf p_1} < {\bf p_2} \ldots < {\bf p_r} \leq n-k$. The singular locus  $X(w)$ consists of $r-1$ components. The components are given by the Schubert varieties corresponding to the  Young diagrams 
${\cal{Y}}(w_1), {\cal{Y}}(w_2), \ldots, {\cal{Y}}(w_{r-1})$  where the sequences ${\bf w_i}$ are given by
\[ {\bf w_i} = (p_1^{q_1},\ldots, p_{i-1}^{q_{i-1}},(p_i-1)^{q_i+1}, p_{i+1}^{q_{i+1}-1},p_{i+2}^{q_{i+2}},\ldots,p_r^{q_r}), \]
for 	$1 \leq i \leq r-1$ and $1 \leq p_i < p_{i+1}$.
\end{theorem}
A box in a Young diagram is said to be a {\em peak} if it does not have any box to its {\em south} or {\em east}. A box is said to be a {\em valley} if there are boxes in its south and east but there is no box to its {\em southeast} in the Young diagram. An easy way to remember description of the irreducible components of the singular locus of $X(w)$ is as follows: they are the Schubert varieties in correspondence with Young diagram ${\cal{Y}}(w_i)$ obtained from ${\cal{Y}}(w)$ by removing the hook which consists of two adjacent peaks and one valley.

\begin{example} Let us consider the Schubert variety $X((3,5,7,9))$ in $\grass{4}{9}$.  The Young diagram corresponding to this Schubert variety looks like 
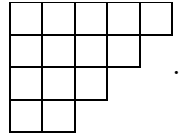
\[ \yng(5,4,3,2).\]
The singular locus obtained by removing the hooks consists of Schubert varieties $X((2,3,7,9)),X((3,4,5,7))$ and $X((3,5,6,7))$. The corresponding Young diagrams are given by
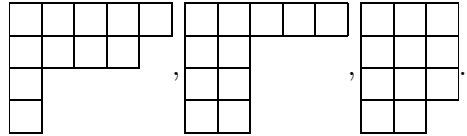
\[ \yng(5,4,1,1), \yng(5,2,2,2) , \yng(3,3,3,2) .\]  
\end{example}
 
We can similarly give a description of the singular locus of opposite Schubert varieties in the Grassmannian. Let $v = (v_1,v_2,\ldots, v_k) \in I(k,n)$. Let $v'_i = n + 1 - v_{k-i}$. Let $v' = (v'_1,v'_2,\ldots, v'_k)$. We first observe that $v' \in I(k,n)$. Let $w_0$ be the unique longest element of $W$. We recall from \cite{kreiman2002richardson} that since $X^v$ is isomorphic to $w_0X(v)$, we have $X^v$ is isomorphic to the Schubert variety $X(v')$. 

We associate an {\em opposite Young diagram} ${\cal{Y}}^{v}$ by removing the Young diagram ${\cal{Y}}({v})$ from the $k \times (n-k)$ rectangle. So the opposite Young diagram ${\cal{Y}}^{v}$ will have $n-k-v_i+i$ boxes in the $i$-th row from bottom. A box in an opposite Young diagram is said to be a {\em peak} if it doesnot have any box to its {\em north} or {\em west}. A box is said to be a {\em valley} if there are boxes in its north and west and if there is no box to its {\em northwest} in the opposite Young diagram. The irreducible components of the singular locus of the opposite Schubert variety $X^v$ are the opposite Schubert varieties $X^v$ such that the associated opposite Young diagram ${\cal{Y}}^{v_i}$ are obtained from ${\cal{Y}}^{w}$ by removing the hook which consists of two adjacent peaks and one valley.

\begin{example} Let us consider the opposite Schubert variety $X^{(2,4,5,7)}$ in $\grass{4}{9}$. We first mark the Young diagram ${\cal{Y}}(w)$ in red inside the $4 \times 5$ rectangle. 
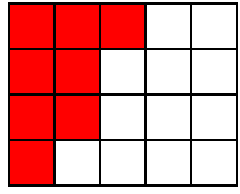
\[
\begin{ytableau}
       \none & *(red)& *(red) & *(red) &  & \\
      \none & *(red) &*(red) & &  &\\
           \none &*(red) &*(red) & & & \\
      \none &*(red) & & & &
\end{ytableau} \] 

After removing the red rectangle, we obtain the following opposite Young diagram ${\cal{Y}}^{v}$

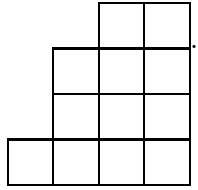
\[
\hspace{1 cm}\ydiagram{2+2,1+3,1+3,4}.
\]

The irreducible components of the singular locus obtained by removing the hooks are the opposite Schubert varieties $X^{(4,5,6,7)}, X^{(2,4,7,8)}$ whose opposite Young diagrams are given by the following 
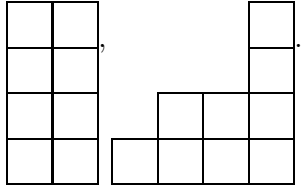
\[\ydiagram{2+2,2+2,2+2,2+2}, \ydiagram{3+1,3+1,1+3,4} .\] 

\end{example}

We now give a description for the singular locus of Richardson varieties in the Grassmannian in terms of diagrams. Using Kleiman's transversality theorem \cite{kleiman1974transversality}, Billey--Coskun characterized the singular locus of Richardson varieties. We recall the following theorem from Billey--Coskun \cite{billey2012singularities}, which we are going to use in this paper. 

\begin{theorem}\cite[Corollary 2.10 ]{billey2012singularities} \label{billeycouskin}
Let $X^v_w$ be a non empty Richardson variety in $\grass{k}{n}$. Let $X(w)_{sing}$ and $X^v_{sing}$ denote the singular loci of the Schubert variety $X(w)$ and the opposite Schubert variety $X^v$ respectively. Then the singular locus of $X^v_w$ is the union of the Richardson varieties 
\[ (X^v_w)_{sing} = (X(w)_{sing} \cap X^v) \cup (X^v_{sing} \cap X(w)). \]
\end{theorem}

Let $w = (w_1, w_2, \ldots w_k)$ and $v = (v_1,v_2,\ldots v_k) \in I(k,n)$ be such that $v \leq w$. Let $X^v_w$ be a Richardson variety. We have associated a Young diagram ${\cal{Y}}(w)$ to the Schubert variety $X(w)$ and an opposite Young diagram ${\cal{Y}}^v$ to the opposite Schubert variety $X^v$. We now associate a {\em skew Young diagram ${\cal{Y}}^v_w$} to the Richardson variety $X^v_w$. It is obtained by removing ${\cal{Y}}^v$ from ${\cal{Y}}_w$. A box in a skew Young diagram is said to be a {\em peak} if either it does not have any box to its {\em south} or {\em east} or it doesnot have any box to its {\em north} or {\em west}.  A box in a skew Young diagram is said to be a {\em valley} if either there are boxes in its south and east and it has no box in its {\em southeast} or if there are boxes in its north and west and  it has no box in its {\em southeast} in the skew Young diagram. The irreducible components of the singular locus of the Richardson variety $X^v_w$ are the Richardson varieties $X^{v_i}_{w_{i}}$ such that the associated skew Young diagram ${\cal{Y}}^{v_i}_{w_i}$ are obtained from ${\cal{Y}}^{w}$ by removing the hook which consists of two adjacent peaks and one valley. 

\begin{example}
Let $X(w), X^v$ and $X^v_w$ denote the Schubert variety, opposite Schubert variety and the Richardson variety in $\grass{4}{9}$ for $v = (2,4,5,7)$ and $w = (3,5,7,9)$. The Young diagram ${\cal{Y}}(w)$ corresponding to the Schubert variety $X(w)$ is the subdiagram of $4 \times 5$ rectangle consisting of the blue and red boxes. The opposite Young diagram  ${\cal{Y}}^{v}$ corresponding to the opposite Schubert variety $X^v$ is the subdiagram consisting of the blue and green boxes. And the skew Young diagram ${\cal{Y}}^v_w$ corresponding to the Richardson variety $X^v_w$ corresponds to the region enclosed by the blue boxes

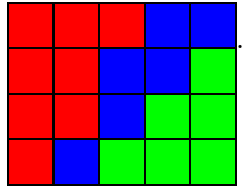
\[
\begin{ytableau}
       \none & *(red)& *(red) & *(red) & *(blue)  & *(blue)\\
      \none & *(red) &*(red) & *(blue)& *(blue) & *(green) \\
           \none &*(red) &*(red) &*(blue)& *(green) &*(green) \\
      \none &*(red) &*(blue) &*(green)&*(green)&*(green)
\end{ytableau}. \] 

We first note that there are three Schubert varieties which are in the singular locus of the Schubert variety $X(w)$ namely $X((2,3,7,9)), X((3,4,5,9))$ and $X((3,5,6,7))$ and  two opposite Schubert varieties which are in the singular locus of the opposite Schubert variety $X^{(4,5,6,7)}$ and $X^{(2,4,7,8)}$. So the singular locus of the Richardson variety is the union of the following Richardson varieties $X^{(2,4,5,7)}_{(2,3,7,9)}, X^{(2,4,5,7)}_{(3,4,5,9)}, X^{(2,4,5,7)}_{(3,5,6,7)},X^{(4,5,6,7)}_{(3,5,7,9)}$ and $X^{(2,4,7,8)}_{(3,5,7,9)}$. But we observe that $X^{(2,4,5,7)}_{(2,3,7,9)}$ and $X^{(4,5,6,7)}_{(3,5,7,9)}$ is empty since $(2,4,5,7) \nleq (2,3,7,9)$ and $(4,5,6,7)\nleq (3,5,7,9)$ in the Bruhat order. So the non empty Richardson varieties in the singular locus are $X^{(2,4,5,7)}_{(3,4,5,9)}, X^{(2,4,5,7)}_{(3,5,6,7)}$ and $X^{(2,4,7,8)}_{(3,5,7,9)}$. Indeed, we note that the skew Young diagram associated with $X^v_w$ has three hooks which consists of two adjacent peaks and one valley. The skew Young diagram corresponding to  $X^{(2,4,5,7)}_{(3,4,5,9)}, X^{(2,4,5,7)}_{(3,5,6,7)}$ and $X^{(2,4,7,8)}_{(3,5,7,9)}$ are marked by the blue boxes in the $4 \times 5$ rectangle

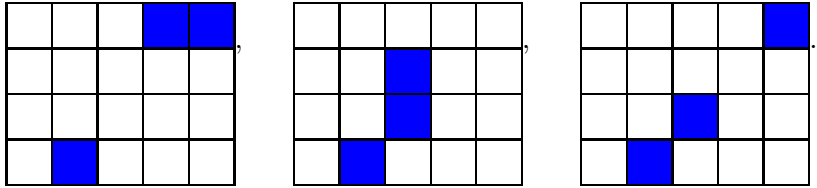
\eat{\[
\begin{ytableau}
       \none & *(red)& *(red) & *(red) & *(blue) &*(blue) \\
      \none & *(red) &*(red) & *(blue)& *(blue)\\
           \none &*(red)   \\
      \none &*(red) 
\end{ytableau}    \begin{ytableau}
       \none & *(red)& *(red) & *(red) & *(blue) &*(blue) \\
      \none & *(red) &*(red)  \\
           \none &*(red) &*(red)   \\
      \none &*(red) & *(blue)
\end{ytableau} 
\begin{ytableau}
       \none & *(red)& *(red) & *(red)  \\
      \none & *(red) &*(red) & *(blue) \\
           \none &*(red) &*(red) & *(blue) \\
      \none &*(red) & *(blue)
\end{ytableau} \]

\[
\ytableausetup{notabloids}
\begin{ytableau}
       \none & \none & \none & \none & *(blue) & *(blue)\\
      \none & \none &\none & *(blue) & *(blue)\\
           \none & \none   \\
      \none & \none 
\end{ytableau}    \begin{ytableau}
       \none &\none& \none & \none & *(blue) &*(blue) \\
      \none & \none & \none  \\
           \none & \none & \none   \\
      \none & \none & *(blue)
\end{ytableau} \begin{ytableau}
       \none & \none & \none & \none  \\
      \none & \none & \none & *(blue)  \\
           \none & \none & \none & *(blue)  \\
      \none &\none & *(blue)
\end{ytableau} \] 
}

\[
\begin{ytableau}
       \none & *(white)& *(white) & *(white) & *(blue)  & *(blue)\\
      \none & *(white) &*(white) & *(white)& *(white) & *(white) \\
           \none &*(white) &*(white) &*(white)& *(white) &*(white) \\
      \none &*(white) &*(blue) &*(white)&*(white)&*(white)
\end{ytableau},
\begin{ytableau}
       \none & *(white)& *(white) & *(white) & *(white)  & *(white)\\
      \none & *(white) &*(white) & *(blue)& *(white) & *(white) \\
           \none &*(white) &*(white) &*(blue)& *(white) &*(white) \\
      \none &*(white) &*(blue) &*(white)&*(white)&*(white)
\end{ytableau}, \begin{ytableau}
       \none & *(white)& *(white) & *(white) & *(white)  & *(blue)\\
      \none & *(white) &*(white) & *(white)& *(white) & *(white) \\
           \none &*(white) &*(white) &*(blue)& *(white) &*(white) \\
      \none &*(white) &*(blue) &*(white)&*(white)&*(white)
\end{ytableau}.
\]
\eat{\[
\ytableausetup{notabloids}
\begin{ytableau}
       \none &\none& \none & \none & *(blue) &*(blue) \\
      \none & \none & \none & \none &*(blue)  \\
           \none & \none & \none   \\
      \none & \none & *(blue)
\end{ytableau},
 \begin{ytableau}
       \none & \none & \none & \none & *(blue)  \\
      \none & \none & \none & \none \\
           \none & \none & *(blue) & \none  \\
      \none & *(blue)
\end{ytableau} \] }

\end{example}

\section{Minimal dimensional Richardson varieties admitting semistable points}

We recall from \cite{bakshi2019torus}, that there exists a unique minimal dimensional Schubert variety $X(w_{k,n})$ in $\grass{k}{n}$ which admits semistable point for $k$ and $n$ coprime. The existence of the unique Schubert variety was first shown by Kannan--Sardar \cite{kannan2009torusA}, however there was a computational error in their description of $w_{k,n}$. Later  Bakshi--Kannan--Subrahmanyam \cite{bakshi2019torus} corrected the error and came up with a simpler proof for the same.

\begin{theorem}\cite[Proposition 2.2]{bakshi2019torus}
Let $k$ and $n$ be coprime. Then $w_{k,n} = (a_1,a_2,\ldots,a_k)$ where $a_i$ is the smallest integer such that $a_i \cdot k \geq i \cdot n$.
\end{theorem}

In \cite{kannan2018torus}, Kannan--Paramasamy--Pattanayak--Upadhyay gave a condition on $v$ and $w$ for which the semistable locus in the Richardson variety $X^v_w$ in $\grass{k}{n}$ is non empty. Bakshi--Kannan--Subhrahmanyam \cite{bakshi2019torus}, showed that there is a unique minimal dimensional Richardson variety  $X_{w_{k,n}}^{v_{k,n}}$ in the $\grass{k}{n}$ which admits a semistable point using standard monomials. We recall their theorem.

\begin{theorem}\cite[Proposition 3.2]{bakshi2019torus}\label{prop3.2}
Let $k$ and $n$ be coprime. Let $v_{k,n}$ be such that $X_{w_{k,n}}^{v_{k,n}}$ is the smallest Richardson variety in $X(w_{k,n})$ admitting semistable points.  Then $v_{k,n} = (1, a_1,\ldots, a_{k-1})$ with the $a_i$ defined as 
the smallest integer satisfying $a_i \cdot k \geq i\cdot n$. 
\end{theorem}

\begin{example} Let us consider our running example of $\grass{4}{9}$. Using the above criteria, one checks that $w_{4,9} = (3,5,7,9)$ and $v_{4,9} = (1,3,5,7)$. The skew Young diagram that one associates with the Richardson variety is the blue region in the rectangle below
\[
\begin{ytableau}
       \none & *(red)& *(red) & *(red) & *(blue)  & *(blue)\\
      \none & *(red) &*(red) & *(blue)& *(blue) & *(green) \\
           \none &*(red) &*(blue) &*(blue)& *(green) &*(green) \\
      \none &*(blue) &*(blue) &*(green)&*(green)&*(green)
\end{ytableau}.\]
\end{example}

\begin{remark} \label{remark}
We note from \ref{prop3.2} that a  Richardson variety $X^v_w$ contains semistable points iff $X^{v_{k,n}}_{w_{k,n}} \subset X^v_w$ iff $v \leq v_{k,n}$ and $w \geq w_{k,n}$. For instance, the Richardson variety $X^{(1,2,4,7)}_{(3,6,7,9)}$ contains semistable points, however $X^{(1,2,6,7)}_{(3,6,7,9)}$ does not contain any semistable point.
\end{remark}

\section{Smooth quotients of Richardson varieties}

Let $k$ and $n$ be coprime. The following lemma is going to be the key for our main theorem. The proof of this lemma follows along the lines described in \cite[Example 3.3]{kannan2014git}.

\begin{lemma} \label{keylemma} Let $v, w \in I(k,n)$. Then  $T \backslash\mkern-6mu\backslash (X_w^v)^{ss}_T({\cal L})$ is smooth if and only if $(X_w^v)^{ss} \subseteq (X_w^v)_{sm}$.
\end{lemma}

\begin{proof}
We first observe that $(X_w^v)^{ss} = (X_w^v)^{s}$, since $gcd(k,n) = 1$ (see \cite[Theorem 3.3]{kannan1998torus},\cite[Corollary 2.5]{skorobogatov1993swinnerton}). Let $T \backslash\mkern-6mu\backslash (X_w^v)^{ss}_T({\cal L})$ be smooth. Let $x \in (X_w^v)^{ss}$. If $x$ is not a smooth point in $X_w^v$ then its image in the quotient will not be a smooth point as well, violating our assumption. 

Conversely, let us assume $(X_w^v)^{ss} \subseteq (X_w^v)_{sm}$. Let $x \in (X_w^v)^{ss}$. Let $x \in B\tau P/P$ for $v \leq \tau \leq w$. Let $\beta_1,\dots, \beta_r$ be a subset of positive roots such that $x = u_{\beta_1}(t_1)\ldots u_{\beta_r}(t_r)\tau P/P$ with $u_{\beta_j}(t_j)$ in the root subgroup $U_{\beta_j}$, $t_j \neq 0$ for $j = 1,\ldots, r$. The isotropy group $T_x$ is $\cap_{i=1}^{i=k} \text{ker}(\beta_j)$. We note from  ~\cite[Example 3.3]{kannan2014git}, $T_x$ is finite and $T_x = Z(G)$, the center of $G$. Working  with the adjoint group we may assume that the stablizer is trivial. So $\schbmodt{w}^{ss}_{T}({\cal L}(n\omega_r))$ is smooth.
\end{proof}
 
We recall from \S3 that $w_{k,n} = (a_1,a_2,\ldots,a_k)$ and $v_{k,n} = (1, a_1,\ldots, a_{k-1})$, where $a_i$ is the smallest integer satisfying $a_i \cdot k \geq i\cdot n$. 
We recall the main result of Bakshi--Kannan--Subrahmanyam \cite{bakshiinprep}. 

\begin{theorem}\cite[Theorem 3.2]{bakshiinprep}\label{bks}
\label{thm:main1}Let $w = (b_1,b_2,\ldots, b_k) \in I(k,n)$ with $b_i \geq a_i$, for all $1 \leq i \leq k$. Let $X(v_1),\ldots, X(v_r)$ be $r$ components in the singular locus of $X(w)$. Then the following are equivalent
\begin{itemize}
 \item[$(1)$] $\schbmodt{w}^{ss}_{T}({\cal L})$ is smooth.
 \item[$(2)$]  We have $ v_i \ngeqslant w_{k,n}$ for all $i$.
 \item[$(3)$] Whenever $b_j \geq b_{j-1}+2$ we have $a_j \geq b_{j-1}+1$.
 \end{itemize} 
\end{theorem}

Our goal in this paper is to extend the above result for Richardson varieties. We now prove the main theorem of this paper. This builds on the proof of \ref{bks} from \cite{bakshiinprep}.

\begin{theorem} \label{maintheorem}
Let $k$ and $n$ be coprime. Let $v = (c_1, c_2,\ldots, c_k)$ and $w = (b_1,b_2,\ldots, b_k)$ be such that $v \leq v_{k,n}$ and $w \geq w_{k,n}$. \eat{ Let $c_1 = 1$. Let $b_i \geq a_i$ for all $1 \leq i \leq n$ and $c_i \leq a_{i-1}$ for all $2 \leq i \leq n$.} Let $X_{w_1}^{v_1},\ldots, X_{w_r}^{v_r}$ be $r$ components in the singular locus of $X(w)$. Then the following are equivalent
\begin{itemize}
 \item[$(1)$] $T \backslash\mkern-6mu\backslash (X_w^v)^{ss}_T({\cal L})$ is smooth.
 \item[$(2)$] We have $ w_i \ngeqslant w_{k,n}$ and $ v_i \nleqslant v_{k,n} $ for all $i$. 
 \item[$(3)$] Whenever $b_j \geq b_{j-1}+2$, we have $a_j \geq b_{j-1}+1$, and whenever $c_j \geq c_{j-1}+2$, we have $ a_{j-1} \leq c_j+1$.
 \end{itemize} 
\end{theorem}

\begin{proof}

We first observe that $(X_w^v)^{ss}_T({\cal L})$ is nonempty since $v \leq v_{k,n}$ and $w \geq w_{k,n}$. We know from \ref{keylemma} that $T \backslash\mkern-6mu\backslash (X_w^v)^{ss}_T({\cal L})$ is smooth if and only if $(X_w^v)^{ss} \subseteq (X_w^v)_{sm}$, where $(X_w^v)_{sm}$ denotes the smooth locus of $(X_w^v)$. So whenever $ w_i \ngeqslant w_{k,n}$ and $ v_i \nleqslant v_{k,n} $, we observe that the singular locus of $(X_w^v)$ does not contain a semistable point. Hence, $T \backslash\mkern-6mu\backslash (X_w^v)^{ss}_T({\cal L})$ is smooth. Conversely, $T \backslash\mkern-6mu\backslash (X_w^v)^{ss}_T({\cal L})$ is smooth implies $(X_w^v)^{ss} \subseteq (X_w^v)_{sm}$, which further implies that the singular locus of $(X_w^v)$ cannot contain a semistable point. Hence, $(1)$ and $(2)$ are equivalent. 

We note from \ref{billeycouskin} that the singular locus of the Richardson variety $X^v_w$ is the union of Richardson varieties which are either of the form $X^{v}_{w'}$ or of the form $X^{v'}_{w}$, where $X(w')$ (respectively, $X^{v'}$) are the Schubert varieties (respectively, opposite Schubert varieties) in the irreducible component of the singular locus of the Schubert variety $X(w)$ (respectively, $X^v$). Let $X^{v}_{w'}$ be a Richardson variety which is in the singular locus of $X^v_w$, where $X(w')$ is the Schubert variety in the singular locus of $X(w)$ obtained by removing a hook at a valley which lies in $j$-th row from bottom of the Young diagram ${\cal{Y}}(w)$. Now from our hypothesis and \ref{remark}, we note that $X^{v}_{w'}$ contains a semistable point if and only if $ X^{v_{k,n}}_{w_{k,n}}\subset X^{v}_{w'} $ if and only if $w_{k,n} \leq w'$. Let $w' = (b'_1,b'_2,\ldots, b'_k)$. Recall $w_{k,n}= (a_1, a_2,\ldots, a_k)$. By assumption, we have $b_j \geq b_{j-1}+2$. Let $P$ denote the set of all $t$ such that $b'_t \neq b_t$. Clearly, for $t > j$ we have $b'_t = b_t$.  Let $P = \{ m, m+1,\ldots, j \}$, where $1 \leq m \leq j$. We have $a_t \leq b'_t$ for all $t \in \{1, 2,\ldots, k\} =\backslash P$. We also have $a_t +2 \leq a_{t+1}$ for all $t$. Therefore, we have $a_t \leq b'_t$ for all $t \in \{1, 2,\ldots, k\}$ if and only if $a_j < b_{j-1}+1 = b'_{j}$. Equivalently, $w_{k,n} \leq w'$ if and only if $a_j < b_{j-1}+1 = b'_{j}$. So, we have proved that $w_i \ngeqslant w_{k,n}$ for all $i$ iff whenever $b_j \geq b_{j-1}+2$, we have $a_j \geq b_{j-1}+1$.

Let $X^{v'}_{w}$ be a Richardson variety which lies in the singular locus of $X^v_w$, where $X^{v'}$ is the opposite Schubert variety in the singular locus of $X^v$ which is obtained by removing a hook at a valley which lies in $j$-th row from bottom of the opposite Young diagram ${\cal{Y}}^{v'}$. As in the previous paragraph, we note that $X^{v'}_{w}$ contains a semistable point if and only if $X^{v_{k,n}}_{w_{k,n}}\subset X^{v'}_{w}$ if and only if $v_{k,n} \geq v'$. Let $v'= (c'_1, c'_2,\ldots, c'_k)$. Recall $v_{k,n}= (1, a_1,\ldots, a_{k-1})$. Let $a_0 = 1$. By assumption, we have $c_j \geq c_{j-1}+2$. In this case, we have $c'_t = c_t$ whenever $t < j$. Let $Q$ denote the set of all $t$ such that $c'_t \neq c_t$. Let $Q = \{ j, j+1,\ldots, m \}$ where $1 \leq m \leq k$. We have $a_{t-1} \geq b'_t$ for all $t \in \{1, 2,\ldots, k\} \backslash Q$. So, we will have $a_{t-1} \geq c'_t$ for all $1\leq t \leq k-1$ if and only if $a_{j-1} \geq c_j+1 = c'_{j}$. Equivalently, $v_{k,n} \geq v'$ if and only if $a_{j-1} \geq c_j+1 = c'_{j}$. Thus, we have proved that $v_i \nleqslant v_{k,n} $ for all $i$ iff whenever $c_j \geq c_{j-1}+2$, we have $ a_{j-1} \leq c_j+1$. This completes the proof.

\end{proof}

\section{Examples and Non-examples}

In this section we discuss some examples of smooth and singular quotients of Richardson varieties in $\grass{4}{9}$.

\begin{example} We know that $X^{(1,3,5,7)}_{(3,5,7,9)}$ is the minimal dimensional Richardson variety admitting a semistable point. Hence, $T \backslash\mkern-6mu\backslash (X^{(1,3,5,7)}_{(3,5,7,9)})^{ss}_T({\cal L})$ is smooth. 
\end{example}

\begin{example} Let us consider the Richardson variety $X^{(1,3,4,6)}_{(3,5,7,9)}$. The singular locus consists of the following Richardson varieties: $X^{(1,3,7,9)}_{(3,5,7,9)}, X^{(3,4,5,6)}_{(3,5,7,9)},X^{(1,3,4,6)}_{(2,3,7,9)} X^{(1,3,4,6)}_{(3,4,5,9)}$ and $X^{(1,3,4,6)}_{(3,5,6,7)}$. Using \ref{maintheorem} we observe that none of the Richardson varieties which lies in the singular locus contains semistable points. Hence, $T \backslash\mkern-6mu\backslash (X^{(1,3,5,7)}_{(3,5,7,9)})^{ss}_T({\cal L})$ is smooth.
\end{example}

\begin{example} Consider the Richardson variety $X^{(1,2,3,5)}_{(3,5,7,9)}$. The singular locus consists of the following Richardson varieties: $X^{(1,2,5,6)}_{(3,5,7,9)}, X^{(1,2,3,5)}_{(2,3,7,9)},X^{(1,2,3,5)}_{(3,4,5,9)}$ and $X^{(1,2,3,5)}_{(3,5,6,7)}$. Note that the Richardson variety $X^{(1,2,5,6)}_{(3,5,7,9)}$ admits semistable points. So $T \backslash\mkern-6mu\backslash (X^{(1,2,4,6)}_{(3,5,7,9)})^{ss}_T({\cal L})$ is not smooth.
\end{example}

\begin{example} Consider the Richardson variety $X^{(1,3,4,6)}_{(5,7,8,9)}$. We observe that the singular locus contains the Richardson variety $X^{(1,3,4,6)}_{(4,5,8,9)}$ which contains semistable points. Hence $T \backslash\mkern-6mu\backslash (X^{(1,3,4,6)}_{(5,7,8,9)})^{ss}_T({\cal L})$ is not smooth either.
\end{example}

\bibliography{main}
\bibliographystyle{plain}

\end{document}